\documentclass{article}
\usepackage{amsmath,amssymb,bbm,color,theorem}
\usepackage[english]{babel}

\newcommand{\CC}{\mathbb{C}}
\newcommand{\RR}{\mathbb{R}}
\newcommand{\NN}{\mathbb{N}}
\newcommand{\fa}{\mathfrak{a}}
\newcommand{\fb}{\mathfrak{b}} 
\newcommand{\fc}{\mathfrak{c}} 
\newcommand{\fg}{\mathfrak{g}}
\newcommand{\cA}{\mathcal{A}}
\newcommand{\cB}{\mathcal{B}}
\newcommand{\cC}{\mathcal{C}}
\newcommand{\cI}{\mathcal{I}}
\newcommand{\cJ}{\mathcal{J}}
\newcommand{\cF}{\mathcal{F}}
\newcommand{\cD}{\mathcal{D}}
\newcommand{\cK}{\mathcal{K}}
\newcommand{\cL}{\mathcal{L}}
\newcommand{\cE}{\mathcal{E}}
\newcommand{\cR}{\mathcal{R}}
\newcommand{\assign}{:=}
\newcommand{\cdummy}{\cdot}
\newcommand{\nobracket}{}
\newcommand{\nocomma}{}
\newcommand{\noplus}{}
\newcommand{\nosymbol}{}
\newcommand{\rank}{\rm{rank}}
\newcommand{\Lin}{\rm{Lin}}

\newcommand{\tmop}[1]{\ensuremath{\operatorname{#1}}}
\newcommand{\tmstrong}[1]{\textbf{#1}}

\newcommand{\tmtextit}[1]{{\itshape{#1}}}

\newenvironment{proof}{\noindent\textbf{Proof.\ }}{\hspace*{\fill}$\Box$\medskip}
\definecolor{grey}{rgb}{0.75,0.75,0.75}
\definecolor{orange}{rgb}{1.0,0.5,0.5}
\definecolor{brown}{rgb}{0.5,0.25,0.0}
\definecolor{pink}{rgb}{1.0,0.5,0.5}
\newtheorem{lemma}{Lemma}[section]
\newtheorem{definition}[lemma]{Definition} 
\newtheorem{proposition}[lemma]{Proposition}
{\theorembodyfont{\rmfamily}\newtheorem{remark}[lemma]{Remark}}
\newtheorem{theorem}[lemma]{Theorem}
\begin{document}

\title{Flat extensions in $\ast$-algebras}
\author{B. Mourrain \& K. Schm{\"u}dgen}
\maketitle
\abstract{ The main result of the paper is a flat extension theorem
  for positive linear functionals on $*$-algebras. The theorem is
  applied to truncated moment problems on cylinder sets, on matrices
  of polynomials and on enveloping algebras of Lie algebras.}

\medskip
\noindent{}\textbf{Keywords:}  $\ast$-algebras; flat extension;
hermitian linear form; GNS construction; truncated moment problem; enveloping algebra;

\section{Introduction}

Given real numbers $s_\alpha$, where $\alpha \in \NN_0^d$, $|\alpha|\leq 2n$, the (real) truncated moment problem \cite{CF91,CF96} asks: When does there exist a positive Borel  measure $\mu$ on $\RR^d$ such that $s_\alpha =\int x^\alpha d\mu$ for all $\alpha \in \NN_0^d, |\alpha|\leq 2n$?
A fundamental  result on  this problem is the flat extension theorem due to R. Curto and L. Fialkow \cite{CF96}, see e.g. \cite{L} for a  nice exposition and \cite{LM} for another proof.  
The flat extension theorem extends the corresponding linear functional
on the polynomials of degree at most $2n$ to a flat positive linear
functional on the whole polynomial algebra $\CC[x_1,\dots,x_d]$. Then the extended functional  is represented 
by a positive measure. For this  it is crucial that the extension is also flat, because then the  GNS representation acts on a finite dimensional space and the existence of the measure $\mu$ can be derived (for instance) from the finite dimensional spectral theorem.
The aim of this paper is to extend the flat extension theorem  to complex unital $*$-algebras. 

Suppose that $\cA$ is a unital complex $*$-algebra. Let $\pi$ be a $\ast$-representation of $\cA$ on a unitary space $(V,
\langle \cdot, \cdot \rangle)$.  For any $v \in V$, there is a  positive linear functional $L_v$
on $A$ defined by 
$$L_v (a) = \langle \pi (a) v, v \rangle,\quad a \in
\cA.$$
Functionals of  this form are called {\it vector functionals} in the representation $\pi$. 

Let $L$ be a linear functional on a $*$-invariant linear subspace $\cB$ of $\cA$. A natural question is when $L$ has an extension to a positive linear functional $\tilde{L}$ on the whole $*$-algebra $\cA$. By the GNS construction (see e.g. \cite[Section 8.6]{Sch1990}) this holds if and only if $L$ is the restriction of a vector functional $L_v$ in some representation $\pi$. Furthermore,  the extension $\tilde{L}$ should be "nice" in order to represent it in some "well-behaved" representation $\pi$. The latter  is the counter-part for requiring that the extended functional can be given by some positive measure.
In general, it seems that these problem are  difficult and no simple answers can be expected. 

The main result of this paper is  a  general flat extension theorem
for hermitean linear functionals. Roughly speaking and omitting
technical details and assumptions, it says that if $L$ is a hermitean
linear functional  on a subspace which is flat with respect to some
appropriate smaller subspace, then $L$ has a unique flat extension
$\tilde{L}$ to $\cA$. Furthermore, if $L$ is positive, so is
$\tilde{L}$.  

This theorem contains the Curto-Fialkow theorem as a special case when $\cA$ is the polynomial algebra. But it applies also to certain noncommutative $*$-algebras and it allows one to treat  noncommutative truncated moment problems. We mention  three applications, the first concerns truncated moment problems on cylinder sets, the second is about truncated moment problems for matrices over polynomial algebras,    and the third deals with enveloping algebras of finite dimensional Lie algebras. 

Flat extensions of positive functionals on path $*$-algebras have been  considered in \cite{P}.

\section{Basic definitions and preliminaries}

First we prove a proposition which contains a reformulation of flatness for hermitean matrices. It  gives the justifaction for the flatness Definition \ref{defflatfunc} below.

Let $A \in M_n ( \CC)$, $B \in M_{n, k} ( \CC)$, $C \in M_k (
\CC)$, where $A = A^{\ast}$ and $C = C^{\ast}$, and consider the
hermitean block matrix $X$ acting on the vector space $\CC^{n + k}$:
\[ X = \left( \begin{array}{ll}
     A & B\\
     B^{\ast} & C
   \end{array} \right) . \]
Then the matrix $X$ is called a \tmtextit{flat extension} of the matrix $A$
if $\rank A =\rank X$.

\begin{proposition}\label{matrixcase}
  The following statements are equivalent:
  \begin{itemize}
  \item[(i)] $X$ is a flat extension of $A$.
  \item[(ii)] $\CC^{n + k} = ( \CC^n, 0) + \ker X$.
  \end{itemize}
\end{proposition}

\begin{proof}
  (i)$\to$(ii): Since $X$ is a flat extension of $A$, a classical result of
  Shmuljan {\cite{Sml}} states that there exist a matrix $W \in M_{n, k} (
  \CC)$ such that $B = AW$ and $C = W^{\ast} AW$. Let $x \in
  \CC^n$ and $y \in \CC^k$. Then we compute $X (- Wy, y) = 0$,
  so that $(x, y) = (x + Wy, 0) + (- Wy, y) \in \CC^{n + k} + \ker X$
  which proves (ii).
  
 (ii)$\to$(i): 
By reordering the canonical basis of $\CC^{n}$, we may assume that the
first $n'$ canonical basis elements of $\CC^{n}$ are linearly
independent of $\ker X$ and that, together with a basis of $\ker X$, they
form a basis of $\CC^{n+k}$. 
This implies that $\CC^{n+k}= (\CC^{n'},0)\oplus \ker X$ and we can take
$n'=n$ in order to prove that (ii) implies (i). 

The matrix of change of bases between
the canonical basis of $\CC^{n+k}$ and this basis is of the form
$$ 
P= \left (
\begin{array}{cc}
I & -W \\
0 & I
\end{array}
\right)
$$
and the matrix of the symmetric operator in this basis is 
\begin{eqnarray*}
\left (
\begin{array}{cc}
A & 0 \\
0 & 0
\end{array}
\right)
&=& 
\left (
\begin{array}{cc}
I & 0\\
-W^{\ast} & I
\end{array}
\right)
 \left( \begin{array}{cc}
     A & B\\
     B^{\ast} & C
   \end{array} \right)
\left (
\begin{array}{cc}
I & -W \\
0 & I
\end{array}
\right)\\
&=&
\left (
\begin{array}{cc}
A & B-A\, W \\
B^{\ast}-W^{\ast}\,A & C-W^{\ast}B -B^{\ast}W+ W^{\ast} A \, W
\end{array}
\right)
\end{eqnarray*}
This implies that $B= A\, W$, $B^{\ast}= W^{\ast} A$ and $C= W^{\ast}A\,
W$. Thus $X$ is a flat extension of $A$.
%
% For let
%   $x \in \CC^n$ and $y \in \ker X$. Then $X (x + y) = Xx = (Ax,
%   B^{\ast} x)$. Since $X \succcurlyeq 0$, we conclude that $Ax = 0$ implies
%   $B^{\ast} x = 0$ and hence $X (x + y) = 0$. Therefore, since $\CC^{n
%   + k} = \CC^n + \ker X$ by (ii), there is a well-defined linear map $T
%   : \mathrm{ran} A : \to \mathrm{ran} X$ given by $T (Ax) = X (x + y)$. This
%   implies that $\mathrm{rank} \hspace{0.25em} A = \dim \mathrm{ran}
%   \hspace{0.25em} A \leq \dim \mathrm{ran} \hspace{0.25em} X = \mathrm{rank}
%   \hspace{0.25em} X$. Since obviously $\mathrm{rank} A \leq \mathrm{rank}
%   \hspace{0.25em} X$, we get $\mathrm{rank} \hspace{0.25em} A = \mathrm{rank}
%   \hspace{0.25em} X$, that is, $X$ is a flat extension of $A$.
\end{proof}

Throughout this paper we assume that $\mathbb{\cA}$ is a complex
unital $\ast$-algebra. The involution of $\cA$ is denoted by $\ast$
and the unit element by $1$.

Let $\cC$ be a $\ast$-invariant linear subspace of
$\cA$, which contains the unit element $1$ of $\cA$. A linear
functional $L$ on $\cC$ is called \tmtextit{hermitian} if $L
(b^{\ast}) = \overline{L (b)}$ for all $b \in \cC$. Set
\[ \cC^2 \assign {\Lin} \, \{ ab ; a, b \in \cC \}. \]
Now suppose that $L$ is a hermitian linear functional on $\cC^2$. Then
\[ \langle a, b \rangle_L \assign L (b^{\ast} a), a, b \in \cC, \]
defines a hermitian sesquilinear form $\langle \cdummy, \cdummy
\rangle_L$ acting on the vector space $\cC$.
We denote by $K_L ( \cC)$  the vector space
\[ K_L ( \cC) \assign \{a \in \cC : \langle a, b \rangle_L = 0
   \nocomma, \forall b \in \cC \} . \]
We verify that $K_L ( \cC)$ is the kernel of the Hankel operator 
$H_L : \cC  \rightarrow  \cC^{\star}$ defined by
\begin{eqnarray*}
  H_L :  a \mapsto  a \star L,  \quad 
  {\rm where}\quad a \star L : b \in \cC \mapsto L (b^{\ast} a) \in
  \CC.
\end{eqnarray*}
Further, for $a \in \cC$, we define
\[ \mathcal{D_C} (a) = \{b \in \cC^{2} : ab \in \cC^2 \},
   \hspace{1em} \rho (a) b = ab \hspace{1em} \mathrm{for} \hspace{1em} b \in
   \cD_{\cC} (a) . \]
Then $\rho (a)$ is a linear operator with domain $\mathcal{D_C} (a)$ on
$\cC$. It is easily verified that
\[ \begin{array}{lll}
     \rho  (ab) c & = & \rho (a) \rho (b) c \hspace{1em} 
     \mathrm{if} \hspace{1em} abc \in \cC^{2},\\
     \rho (1) b & = & b \hspace{1em} \mathrm{\tmop{for}}
     \hspace{1em} b \in \cC^{2},\\
     L (c^{\ast} ab) & = & \langle \rho_{\nosymbol} (a) b, c \rangle_L =
     \langle b, \rho (a^{\ast}) c \rangle_L \hspace{1em}
     \mathrm{if} \hspace{1em} c^{\ast} ab \in \cC^{2} .
   \end{array} \]
A linear functional $L$ on $\cC^2$ is called
\tmtextit{positive} if $L (a^{\ast} a) \geqslant 0$ for all $a \in
\cC$.

\begin{lemma}
  If a functional $L$ on $\cC^2$ is positive, then $K_L (
  \cC) = \{a \in \cC : L (a^{\ast} a) = 0\}$.
\end{lemma}

\begin{proof}
  If $a \in K_L \mathcal{( C)}$, then in particular $L (a^{\ast} a) = \langle
  a, a \rangle_L = 0$. Conversely, suppose that $L (a^{\ast} a) = 0$. Since
  $L$ is positive, we have the Cauchy-Schwarz inequality
  \[ |L (b^{\ast} a) |^2 \leq L (a^{\ast} a) L (b^{\ast} b) \]
  for $b \in \cC$, which implies that $\langle a, b \rangle_L = L
  (b^{\ast} a) = 0$, so that $a \in K_L ( \cC)$.
\end{proof}

The following definition contains the main  concept of this paper. It is  suggested by the equivalence of conditions (i) and (ii) in Proposition \ref{matrixcase}.
\begin{definition}\label{defflatfunc}
  Let $\cB$ and $\cC$ be $\ast$-invariant linear subspaces of
  $\cA$ such that $1 \in \cB$ and $\cB \subseteq
  \cC$. A hermitian linear functional $L$ on $\cC^2$ is called
  a {\rm flat extension} with respect to $\cB$ if
  \[ \cC = \cB + K_L ( \cC) . \]
\end{definition}
Another important notion (see e.e. \cite[Definition 8.1.9]{Sch1990}) is the following.
\begin{definition}\label{defrep}
  Let $V$ be a vector space equipped with a hermitan sesquilinear form
  $\langle \cdummy, \cdummy \rangle$. A {\rm $\ast$-representation} of $\cA$
  on $(V, \langle \cdot, \cdot \rangle)$ is an algebra homomorphism $\rho$ of
  $\cA$ into the algebra of linear operators on $V$ such that
  $\rho (1) v = v$ and $ \langle \rho (a) v, w \rangle = \langle v, \rho
  (a^{\ast}) w \rangle  \hspace{1em} \mathrm{for} \hspace{1em} v, w \in V.$
\end{definition}

Now suppose that $L$ is a hermitian linear functional defined on the whole
$\ast$-algebra $\cA$. Then it is easily checked that $K_L (\cA)$ is a left ideal.
Therefore, we can quotient out the kernel $\cI = K_L ( \cA)$ and get a \tmtextit{non-degenerate} hermitian form on {$\cA/\cI$}, denoted again by $\langle
\cdummy, \cdummy \rangle_L$ by some abuse of notation, and a
$\ast$-representation $\rho$ on $( \cA/\cI, \langle \cdot, \cdot
\rangle_L)$ such that
\[ L (a) = \langle \rho (a) 1,1 \rangle_L \hspace{1em} \mathrm{for} \hspace{1em}
   a \in \cA . \]
If the functional $L$ is positive, then this procedure coincides with the
"ordinary" GNS construction, see e.g. {\cite[Theorem 8.6.2]{Sch1990}}.

For the formulation of our main theorem  the following concept is convenient.

%Suppose that the algebra $\cA$ is generated by $a_i, i \in I$.

\begin{definition} Let $\{a_i, i \in I\}$ be a fixed set of generators of the algebra $\cA$.
  For a vector space $V \subset \cA$, let $V^+ = {\Lin}\, \{ a_i v|
  \nobracket i \in I, v \in V \}$ be the {\em prolongation} of $V$. We shall say that $V$ is {\rm connected
  to $1$} if $1 \in V$ and there exists a increasing sequence of vector spaces
  \[ V_0 = \CC \cdot 1  \subset V_1 \subset \cdots \subset V_l \subset
     \cdots \subset V \]
  such that $\cup_{l \in \mathbb{N}} V_l = V$ and $V_{l + 1} \subset V_l^+$.
  For an element $v \in V$, the {\rm $V$-index} of $v$ is the smallest $l \in
  \mathbb{N}$ such that $v \in V_l$. Let $V^{[ 0]} = V$ and for $l \in
  \mathbb{N}$, we define $V^{[ l + 1]} = ( V^{[ l]})^+ .$ The index of an arbitrary
  element $a \in \cA$ is its $\CC \cdot 1 $-index.
\end{definition}

\section{The extension theorem}

Throughout this section we suppose that $\{a_i ; i \in I\}$ is a fixed set of hermitean
generators of the unital complex $\ast$-algebra $\cA$, such that 
$\forall i \in I$ there exists $j\in I$ such that
$a_i^{\ast} = a_j$.

Let $\cF = \CC \langle x_i, i \in I \rangle$ denote the free
complex unital $\ast$-algebra in variables $x_i$, $i \in I$. Then there is a $\ast$-homomorphism $\sigma
: \cF \to \cA$ such that $\sigma (x_i) = a_i$, $i \in I$. The
kernel $\cI$ of $\sigma$ is a two-sided $\ast$-ideal of $\cF$
and we have and $\cA = \cF / \cI$. For any 
set of generators $S$ of the two-sided ideal $\cI$, 
we denote by $\delta (S)\in \NN\cup\{\infty\}$ the maximal index of the elements of $S$.
Let 
$\delta (\cA)$ be the lowest $\delta (S)$ of all sets $S$ of
generators of $\cI$.

\begin{theorem}\label{thm:main}
Let $\cB$ and ${\cC}$ be $\ast$-invariant linear
  subspaces of $\cA$, connected to $1$ such that $\cB^{[m]} \subseteq
  \cC$ with $2m \geq \delta (\cA)$.
  Let $L$ be a hermitian linear functional on
  $\cC^2$ which is a flat extension with respect to $\cB$.
  Then there exists a unique extension of $L$ to a hermitian linear functional
  $\cL$ on $\cA$ which is a flat extension with respect to $\cC$. 
  
Let us choose a linear subspace  $\cB'$ of $\cB$ such that $1\in \cB'$
and $\cC$ is the direct sum of the vector spaces $\cB'$ and $K_L ( \cC)$. 
The  form  $\langle \cdot , \cdot \rangle_{\cL}$ is non-degenerate on
$\cB'$ and the projection of $\cC$ on  $\cB'$ along $K_L ( \cC)$  extends to a projection 
$\pi_{\cL} : \cA \rightarrow \cB'$. The
$*$-representation $\rho_{\cL}$  of $\cA$ associated with $\cL$ acts on $(\cB', \langle \cdot , \cdot
\rangle_{\cL})$ by \, $\rho_{\cL}(a)b=\pi_{\cL}(ab)$, $a\in \cA, b \in \cB'$.

Further, if $L$ is a positive linear functional on $\cC^2$, then $\cL$
  is a positive linear functional on $\cA$ and  $\langle \cdot , \cdot \rangle_{\cL}$ is  a scalar product on $\cB'$.
\end{theorem}

The remaining part of this section, is devoted to the proof of this theorem.
Let us recall the main assumptions.
Since the vector space $\cC$ is connected to $1$,  there exists
a sequence of \ vector spaces $\cC_0 \subset \cC_1 \subset \cdots \subset \cC_i \subset
\cdots \subset \cC$ such that $\cC_{i + 1} \subset \cC_i^+$ and $\cup_i \cC_i = \cC$.
%Their image by $\sigma$ is denoted $\cC_i = \sigma ( C_i)$.
%Let $G$ be a generating set of the kernel $\cI$ of $\sigma$. Let $2 m
%\in \mathbb{N}$ be a bound on the index of the elements in $G$. We assume
%that $B^{[ m]} \subseteq C$.
That the hermitean linear functional $L$  on $\cC^2$ is a flat extension means that
\[ \cC =\cB+ K_L ( \cC). \]
%where $K_L ( \cC) \assign \{a \in \cC : \langle a, b \rangle_L
%= 0, \forall b \in \cC \}$.
%}
A first consequence of these hypotheses is the following property:
\begin{eqnarray}
  a \in K_L ( \cC) & \Leftrightarrow & \forall b \in \cC, L
  (b^{\ast} a) = 0  \label{eq:inK1}\\
  & \Leftrightarrow & \forall b \in \cB, L (b^{\ast} a) = 0 
  \label{eq:inK2}
\end{eqnarray}
%Let us choose a linear subspace $\cB'$ of $ \cB$ such that $\cB'$ contains $1$ and $$\cC =\cB' \oplus K_L ( \cC).$$ 
%Let $\pi : \cC \rightarrow \cB'$ be the projection of
%$\cC$ along $K_L ( \cC)$ on 
%$\cB'$ and let $\cK= ( \cB')^+ \cap K_L ( \cC)$.
%Let $\pi : \cC \rightarrow \cB'$ be the projection of
%$\cC$ along $K_L ( \cC)$ on a supplementary space
%$\cB' \subset \cB$. We assume that $\cB'$ contains $1$. Let
%$\cK= ( \cB')^+ \cap K_L ( \cC)$.

\subsection*{Coherence}
The flat extension property induces a structure of $K_L ( \cC)$,
which we are going to analyze.
We first show that $K_L ( \cC)$ is stable by prolongation and
intersection with $\cC$.
\begin{lemma}
  \label{lem:0} $K_L ( \cC)^+ \cap \cC \subset K_L (\cC)$.
\end{lemma}

\begin{proof}
  For any element $\kappa \in K_L ( \cC)^+$, we have
  \[ \kappa = \sum_{j \in J} a_j \kappa'_j, \]
  with $J \subset I$, $\kappa'_j \in K_L ( \cC)$. For any $b \in B$, \
  \begin{eqnarray*}
    L (b^{\ast} k) & = & \sum_{j \in J} L (b^{\ast} a_j \kappa'_j) = 0
  \end{eqnarray*}
  since $b^{\ast} a_j \in \cB^+ \subset \cC$ and $\kappa'_j
  \in K_L ( \cC)$. The relation $\left( \ref{eq:inK2} \right)$ implies
  that $\kappa \in K_L ( \cC)$.
\end{proof}

\begin{remark}
There is no loss of generality to assume that $1 \in \cB'$. Otherwise 
  $1 \in K_L ( \cC)$. By Lemma \ref{lem:0}, $\cC_1 \subset (
  \cC_0)^+ \cap \cC \subset K_L ( \cC)$. By induction,
  we deduce that $\forall l \in \mathbb{N}, \cC_l \subset
  K_L ( \cC)$, which implies that $\cC= \cup \cC_l =
  K_L ( \cC)$. Consequently, $L \equiv 0$. But in this case the
  assertion of Theorem \ref{thm:main} is trivial.
\end{remark}
The next lemma shows that $K_L ( \cC)$ contains iterated
prolongations of $\cK$ intersected with $\cC$. As we will see, $\cK$
 generates the left ideal $K_\cL ( \cC)$ of the extension $\cL$.
\begin{lemma}
  \label{lem:12}$\cK^{[ 2 m - 1]} \cap \cC \subset K_L (
  \cC)$.
\end{lemma}

\begin{proof}
  Let $\kappa \in \cK^{[ 2 m - 1]} \cap \cC \nosymbol$. It is
  a sum of terms of the form $a_{i_1} \cdots a_{i_l} \kappa'$ with $l
  \leqslant 2 m - 1$, $i_k \in I$, $\kappa' \in \cK$. We decompose
  $a_{i_1} \cdots a_{i_l} =\fa_1 \fa_2$ as a product of a
  term $\fa_1$ of index $\leqslant m$ and a term $\fa_2$ of
  index $m - 1$. Then, we have $\forall \fb \in \cB$,
  $\langle \fa \kappa' | \fb \nobracket  \rangle = \langle
  \fa_2 \kappa' | \fa_1^{\ast} \fb \nobracket 
  \rangle$.
  
  By hypothesis, $\cB^{[ m]} \subset \cC$ and $\cK \subset \cB^{[ 1]}$, \
  and $\fa_2 \kappa' \in \cK^{[ m - 1]} \subset
  \cB^{[ m]}\subset \cC$.
  
  Let us prove that $\fa_2 \kappa' \in K_L ( \cC)$: $\forall
  \fb' \in \cB$, $\fa_2^{\ast} \fb' \in
  \cB^{[ m]} \subset \cC$. Relation $( \ref{eq:inK1})$ implies
  that
  \[ \langle \fa_2 \kappa' | \fb \nobracket'  \rangle =
     \langle \kappa' | \fa_2^{\ast} \fb \nobracket'  \rangle
     = 0. \]
  Relation $( \ref{eq:inK2})$ implies that $\fa_2 \kappa' \in K_L (
  \cC)$.
  
  By hypothesis, we also have $\forall \fb
  \in \cB$, $\fa_1^{\ast} \fb \in
  \cB^{[ m]} \subset \cC$ \ and by relation $(
  \ref{eq:inK1})$, 
$$\langle \fa_2 \kappa' | \fa_1^{\ast}
  \fb  \rangle = 0 = \langle \fa \kappa' |
  \fb  \rangle .
$$ 
We deduce that \ $\forall \fb \in \cB$, 
$\langle \kappa | \fb \nobracket  \rangle = 0,$
  which implies, by relation $( \ref{eq:inK2})$, that $\kappa \in K_L (
  \cC)$. This ends the proof of this Lemma.
\end{proof}

\subsection*{The GNS construction}
We describe now a GNS-type construction, which allows to extend $L$ on
$\cA$. For $i \in I$, we define the following operator
\begin{eqnarray*}
  X_i : \cB & \rightarrow & \cB\\
  b & \mapsto & \pi (a_i b)
\end{eqnarray*}
and the map
\begin{eqnarray*}
  \phi : \cF & \xrightarrow{} & \cB\\
  p & \mapsto & p (X) (1).
\end{eqnarray*}
The kernel of $\phi$ is denoted $\cJ$. We easily check that
$\cJ$ is a left ideal of $\cF = \CC \langle x_i, i \in
I \rangle$. Let $B, C\subset \cF$ be $\ast$-invariant linear
  subspaces of $\cF$, connected to $1$ such that $\sigma (B)= \cB$ and
  $\sigma (C)= \cC$. We denote by $S$ a generating set of the two-sided
  ideal $\cI=\ker \sigma$ such that $\delta (S)= \delta (\cA)$ is minimal.

\begin{lemma}
  \label{lem:5}$\forall p, q \in C, \phi (p q) = \sigma ( p) \phi (q) +
  \kappa_1 \phi ( q) + \kappa_0$ with $\kappa_0 \in \cK \subset K_L (
  \cC), \kappa_1 \in K_L ( \cC)$. If l is the index of $p$
  then {\color{red} ${\color{black} \phi (p q) - \sigma ( p) \phi (q) \in
  \cK^{[l - 1]} .}$}
\end{lemma}

\begin{proof}
  We prove the property by induction on the index $l$ of $p$.
  
  The property is obviously true for $l = 0$, in which case we can take $p = 1$
  and $\phi ( 1) = \sigma ( 1) = 1$.
  
  Let us assume that it is true for $l \in \mathbb{N}$ and let $p \in C$ of
  index $l + 1$. Then $\sigma ( p) = \sum_{j \in J} a_j \sigma ( p'_j)$ with
  $p'_j \in C$ of index $\leqslant l$.{\tmstrong{}} By induction hypothesis,
  we have
  \begin{eqnarray*}
    \phi (pq) & = & \sum_{j \in J} \pi ( a_j \phi (p'_j q))\\
    & = & \sum_{j \in J} a_j  \phi (p'_j q) + \kappa_0\\
    & = & \sum_{j \in J} a_j  ( \sigma (p'_j) \phi ( q) + \kappa_j' \phi (
    q)) + \kappa_0\\
    & = & \sigma ( p) \noplus \phi ( q) + \kappa_1 \phi ( q) \noplus +
    \kappa_0
  \end{eqnarray*}
  with $\kappa_0 \nocomma \in \cK \subset K_L ( \cC),
  \kappa_j' \in K_L ( \cC)$ and {\color{black} ${\color{red}
  {\color{black} \kappa_1 = \sum_{j \in J} a_j \kappa_j' \in K_L (
  \cC)^+ \cap \cC}}$}. By Lemma \ref{lem:0}, $\kappa_1 \in K_L
  ( \cC)$. By induction hypothesis, \ $\phi (p'_j q) - \sigma (p'_j)
  \phi ( q) \in \cK^{[ l - 1]}${\tmstrong{}}. Therefore
  \[ \phi (pq) - \sigma ( p) \noplus \phi ( q) = \sum_{j \in J} a_j  ( \phi
     (p'_j q) - \sigma (p'_j) \phi ( q)) + \kappa_0 \in ( \cK^{^{[ l -
     1]}})^+ =\cK^{[ l]} \]
  which proves the induction hypothesis for $l + 1$ and concludes the proof of
  this Lemma.
\end{proof}

{\color{black} \begin{proposition}
  \label{prop:5}$\forall g \in S, \forall b, b' \in B,$ $\phi ( b g b') = 0$.
\end{proposition}

\begin{proof}
  By Lemma \ref{lem:5}, $\phi ( g b') = \sigma ( g) \phi ( b') \noplus +
  \kappa = \kappa$ with $\kappa \in \cK^{[ 2 m - 1]}$. As $\phi ( g
  b') \in \cB' \subset \cB,$ by Lemma \ref{lem:12}, $\kappa
  \in K_L ( \cC) \cap \cB' = \{ 0 \} .$ This implies that
  $\phi ( g b') = 0$. Consequently, $\phi ( b g b') = \phi ( b ( X) \phi ( g
  b')) = 0$
\end{proof}}

\begin{proposition}
  \label{prop:6}$\forall p \in C$, $\phi (p) = \pi (\sigma (p))$.
\end{proposition}

\begin{proof}
  Let $p \in C$. Applying Lemma \ref{lem:5} with $q = 1$, we have $\phi ( p) =
  \sigma ( p) \noplus + \kappa$, $\kappa \in K_L ( \cC)$.
  
  As $\sigma ( p) \in \cC, \phi (p) \in \cB'$ and $\kappa \in
  K_L ( \cC)$, $\phi (p)$ is the projection of $\sigma ( p)$ on
  $\cB'$ along $K_L ( \cC)$. In other words, $\phi (p) = \pi (\sigma ( p))$,
  which proves the proposition.
\end{proof}\\
As $\cI=\ker \sigma$ is the left-right ideal generated by $S$, 
this proposition shows that $\phi$ factors through a map $\overline{\phi}$ on
$\cA=\cF/\cI$. 
We prove now that $\ker \overline{\phi}$ is the left ideal generated by $\cK$.
\begin{proposition}
  \label{prop:13}$\sigma ( \ker \phi) = \cA \cdummy \cK$.
\end{proposition}

\begin{proof}
  Let us prove by induction on the index $l$ of $p$ that $\forall p \in A,
  \phi (p) = \sigma ( p) + \kappa$ with $\kappa \in \cA \cdummy
  \cK$.
  
  If $p$ is of index $0$, then $p = 1$ and $\phi ( 1) = 1 = \sigma ( 1) = 1$.
  
  Let us assume that it is true for $l \in \mathbb{N}$ and let $p \in A$ of
  index $l + 1$. Then $p = \sum_{j \in J} x_j p'_j$ with $p'_j \in A$ of index
  $\leqslant l$.{\tmstrong{}} Then we have,
  \begin{eqnarray*}
    \phi (p) & = & \sum_{j \in J} \pi ( a_j \phi (p'_j))\\
    & = & \sum_{j \in J} a_j  \phi (p'_j) + \kappa_0\\
    & = & \sum_{j \in J} a_j  ( \sigma (p'_j) + \kappa_j') + \kappa_0\\
    & = & \sigma ( p) \noplus + \kappa_{}
  \end{eqnarray*}
  with $\kappa_0 \nocomma \in \cK$ and $\kappa = \sum_{j \in J} a_j
  \kappa_j' + \kappa_0 \in \cA \cdummy \cK$. This proves the
  induction hypothesis for $l + 1$.
  
  We deduce that if $a \in \ker \phi$, then \ $\phi ( a)$=0 and $\sigma ( a)
  \in \cA \cdummy \cK$, so that $\sigma ( \ker \phi) \subset
  \cA \cdummy \cK$.
  
  Conversely, $\forall \fa \in \cA, \forall \kappa \in
  \cK \overset{}{} \nocomma$, there exists $a \in \cF$, $k \in
  C$ such that $\sigma ( a) =\fa$ and $\sigma ( k) = \kappa$. By
  Proposition \ref{prop:6}, $\phi ( k) = \pi ( \kappa) = 0$ and $\phi ( a k) =
  a ( X) \phi ( k) = 0$. This shows that $a k \in \ker \phi$ and $\fa
  \kappa = \sigma ( a k) \in \sigma ( \ker \phi) .$ Therefore $\cA
  \cdummy \cK \subset \sigma ( \ker \phi)$.
\end{proof}

\subsection*{The extension of $L$}

We define the extension$\overset{}{}$ $\mathcal{L}$ of $L$ on $\cF$ by
\begin{eqnarray}\label{eq:Lext}
  \mathcal{L}: \cF & \rightarrow & \CC\\ \nonumber
  a & \mapsto & L (\phi (a)).
\end{eqnarray}
We recall that $\cA=\cF/\cI$ where $\cI$ is the left and right ideal degenerated by the elements of
$S$. By Proposition \ref{prop:5}, $\phi ( \cI) = 0$ and
{$\mathcal{L}$} induces a linear form on
$\cA=\cF/\cI$, that we still denote $\mathcal{L}$. If
$\fa= \sigma (a)$ with $a \in \cF$, then $\mathcal{L}
(\fa) = L (\phi (a)) = L (a (X) (1))$.

We first check that $\cL$ is an extension of $L$.
\begin{proposition}
  \label{prop:8}$\forall \fa, \fb \in \cC,
  \mathcal{L} (\fa\fb) = L (\fa\fb)$.
\end{proposition}
\begin{proof}
  Let $\fa= \sigma ( a), \fb= \sigma ( b) \in \cC$
  with $a, b \in C$. By Proposition {\ref{prop:6}}, 
\begin{eqnarray*}
\fa&=&
  \pi ( a) + \kappa = \phi ( a) + \kappa, \\
\fb&=& \pi ( a) + \kappa' =
  \phi ( b) + \kappa'
\end{eqnarray*}
with $\kappa, \kappa' \in K_L ( \cC)$. Relation \eqref{eq:inK1} implies
  have 
$$L ( \fa\fb) = L ( ( \phi ( a) + \kappa) ( \phi ( b)
  + \kappa')) = L ( \phi ( a) \phi ( b)).
$$ 
We can therefore assume that $\fa= \sigma ( a) = \phi (
  a), \fb= \sigma ( b) = \phi ( b) \in \cB$ with $a, b \in B$
  in order to prove that $\mathcal{L} (\fa\fb) = L ( \phi (
  a b)) = L ( \fa\fb)$.
  
  By Lemma \ref{lem:5}, $\phi ( a b) =\fa \phi ( b) + \kappa_1 \phi (
  b) + \kappa_0$ with $\kappa_0, \kappa_1 \in K_L ( \cC)$. By
  Proposition $\ref{prop:6}$, $\phi ( b) = \pi ( \fb) =\fb+
  \kappa'$ with $\kappa' \in K_L ( \cC)$. We deduce that
  \[ \phi ( a b) =\fa\fb+\fa \kappa' + \kappa_1
     \phi ( b) + \kappa_0 . \]
  The relation $\left( \ref{eq:inK1} \right)$ implies that $L ( \fa
  \kappa' + \kappa_1 \phi ( b) + \kappa_0) = 0$, which proves that
  $\mathcal{L} ( \fa\fb) = L ( \phi ( a b)) = L (
  \fa\fb)$.
\end{proof}\\
As a consequence of the previous results, we deduce that $\cL$ is a
flat extension on $\cA$ with respect to $\cB$.
\begin{proposition}\label{prop:sum}
  $K_{\cL}(\cA) = \cA \cdot \cK$ and $\cA
  =\cB' \oplus K_{\cL} (\cA)$.
\end{proposition}
\begin{proof}
  Let $\kappa \in \cK$. $\forall \fa \in \cA$, there
  exists $a \in \cF$, such that $\sigma (a) =\fa$. We deduce
  that
  \[ \cL ( \fa \kappa) = L ( \tilde{\phi} (\fa^{}
     \kappa)) = L ( a (X) \tilde{\phi} (\kappa)) = 0 \]
  since $\tilde{\phi} (\kappa) = \pi ( \kappa) =$0 by Proposition \ref{prop:6}.
  We deduce that $\kappa \in K_{\cL} (\cA)$. As $K_{\cL} (\cA)$ is a
  left ideal, we have $\cA \cdot \cK \subset K_{\cL} (\cA)$.
  
  Conversely, let $\kappa \in K_{\cL} (\cA)$. Then $\forall
  \fa \in \mathcal{\cC}$, \ there exists $a \in \cC$,
  such that $\sigma (a) =\fa$. By proposition \ref{prop:6},
  $\tilde{\phi} \circ \tilde{\phi} ( \kappa) = \pi ( \tilde{\phi} ( \kappa)) =
  \tilde{\phi} ( \kappa)$ and by Proposition \ref{prop:8}, we have
  \[ \mathcal{L} (\fa \kappa) = L ( a ( X)  \tilde{\phi} ( \kappa)) = L
     ( \tilde{\phi} ( \fa \tilde{\phi} ( \kappa))) =\cL (\fa \phi (\kappa)) = L (\fa \phi (\kappa)) = 0. \]
  This implies that $\phi (\kappa) \in K_L ( \cC) \cap \cB' =
  \{0\}$. We deduce that $\kappa \in \ker \tilde{\phi} = {\cA
  \cdot \cK}$ (by Proposition \ref{prop:13}). This proves
  the reverse inclusion $K_{\cL} (\cA) \subset \cA \cdot \cK$.
  
  Let $\fa \in \cA$. It can be decomposed as
  \[ \fa= \tilde{\phi} (\fa) +\fa- \tilde{\phi}
     (\fa), \]
  with $\tilde{\phi} (\fa) \in \cB'$ and $\fa-
  \tilde{\phi} (\fa) \in \ker \tilde{\phi} = \cA \cdot
  \cK= K_{\cL} (\cA)$. This shows that $\cA
  =\cB' + K_{\cL} (\cA)$.
  
  To prove that the sum is direct, we consider an element $\fb \in
  \cB' \cap K_{\cL} (\cA)$. Then $\tilde{\phi} (\fb)
  =\fb$ by Proposition \ref{prop:6} and $\tilde{\phi} (\fb) =
  0$ since $\fb \in K_{\cL} (\cA) = \cA \cdot
  \cK= \ker \tilde{\phi}$. We deduce that $\fb= 0$ and
  therefore that $\cA = B \oplus K_{\cL} (\cA)$.
\end{proof}\\
The next proposition shows the uniqueness of the extension.
\begin{proposition}
There exists a unique hermitian linear form $\cL\in \cA^{\ast}$ which extends $L$ and is a flat
extension with respect to $\cB$.
\end{proposition}
\begin{proof} 
If $\cL$ extends $L$ and is a flat extension with respect to $\cB$
then $\cA =\cB + K_{\cL} (\cA)$.
As $\cL$ is hermitian, $K_{\cL} (\cA)^{\ast}=K_{\cL} (\cA)$.
Then $\forall \kappa \in \cK, \forall
\fa \in \cA$, $\fa=\fb + \kappa'$ with $\kappa'\in K_{\cL} (\cA)$ and 
$$ 
\cL (\fa \, \kappa) = \cL (\fb \, \kappa) + \cL (\kappa' \, \kappa) 
= \cL (\fb \, \kappa) + \overline{\cL (\kappa^{\ast} \, \kappa'^{\ast})} =
L (\fb \, \kappa) = 0,
$$
since $\kappa'^{\ast}\in K_{\cL} (\cA)$, $\fb\, \kappa \in \cC^{2}$
and $\kappa \in K_{L} (\cC)$. 
This proves that $\cA \cdot \cK \subset  K_{\cL} (\cA)$.

For any $\fa \in \cA$, let $\fb=\tilde{\phi} (\fa)\in \cB$ so that $\fa-\fb
\in \cA \cdot \cK \subset  K_{\cL} (\cA)$. 
This implies that
$$ 
\cL (\fa) = \cL (\fb) = L (\fb)= L (\tilde{\phi} (\fa)).
$$
This shows that $\cL$ coincides with the linear form defined in
\eqref{eq:Lext} and thus that the extension $\cL$ is unique.
\end{proof}\\
Finally, we show that the positivity of $L$ can also be extended to $\cL$.
\begin{proposition}
  If $\forall \fa \in \cC, L (\fa^{\ast}
  \fa) {\geqslant} 0$, then $\forall \fa \in
  \cA, \mathcal{L} (\fa^{\ast} \fa) \geqslant 0$.
\end{proposition}

\begin{proof}
  $\forall \fa \in \cA$, let $\fb= \tilde{\phi}
  (\fa)$ so that $\fa-\fb \in \ker \tilde{\phi} =
  \cA \cdot \cK= K_{\cL} (\cA)$. As $\tilde{\phi}
  (\fa^{\ast}) = \tilde{\phi} (\fa)^{\ast}$ and $\forall
  \fb \in \cA$, $\tilde{\phi} (\fa\fb)^{\ast}
  = \tilde{\phi} (\fb^{\ast} \fa^{\ast})$, we deduce from
  Proposition \ref{prop:8} that
  \begin{eqnarray*}
    \mathcal{L} (\fa^{\ast} \fa) & = & \mathcal{L}
    ((\fb^{\ast} + (\fa^{\ast} -\fb^{\ast}))
    (\fb+ (\fa-\fb))) =\mathcal{L}
    (\fb^{\ast} \fb) +\mathcal{L} ((\fa^{\ast}
    -\fb^{\ast})\fb)\\
    & = & \mathcal{L} (\fb^{\ast} \fb) +
    \overline{\mathcal{L} (\fb^{\ast} (\fa-\fb))}
    =\mathcal{L} (\fb^{\ast} \fb) = L (\fb^{\ast}
    \fb).
  \end{eqnarray*}
  As $L$ is positive, $\mathcal{L} (\fa^{\ast} \fa) = L
  (\fb^{\ast} \fb) \geqslant 0$ so that {$\mathcal{L}$} is also positive.
\end{proof}

\subsection*{Representation of $\cA$}
We can now construct the representation of $\cA$ on the vector space
$\cB'$, using the
decomposition of Proposition \ref{prop:sum}.
Let denote by $\pi_{\cL}$ the projection operator of $\cA$ on $\cB'$
along $K_{\cL} (\cA)$. By definition of $K_{\cL} (\cA)$,
for all $\fa,\fb\in \cA$, we have $\langle \pi_\cL(\fa) ,\pi_\cL( \fb) \rangle_{\cL}=\cL (\fb^{*}\fa)$. 
For $\fa\in \cA$, 
let $\rho_\cL(\fa) $ denote the linear operator on $\cB'$ defined by 
\begin{align*}
\rho_\cL(\fa)\fb=\pi_\cL(\fa\fb),\quad{\rm where}\quad \fa\in \cA,~\fb\in \cB'.
\end{align*}
%Then we define the representation of $\cA$ on $\cB'$ by
%\begin{align}\label{defpisubl}
%\rho_{\cL}: \cA &\rightarrow\cB',~~~
%  a \mapsto \pi_{\cL} \circ \rho_{\fa}.
%\end{align}
\begin{proposition} $\rho_{\cL}$ is the $*$-representation of $\cA$ on
  $(\cB', \langle \cdot , \cdot \rangle_{\cL})$ which is associated withthe functional $\cL.$
\end{proposition}
\begin{proof}
By construction, $\forall \fa \in \cA$, $\rho_{\cL} (\fa)$ is a
linear operator of $\cB'$. 
We have to prove first that $\forall \fa,\fb\in \cA$, $\forall \fc \in
\cB'$, $\rho_{\cL} (\fa \fb) (\fc) =
\rho_{\cL} (\fa)\circ \rho_{\cL} (\fb) (\fc)$:
\begin{eqnarray*}
\rho_{\cL} (\fa \fb) (\fc)&=& \pi_{\cL} (\fa \fb \fc) =
\pi_{\cL} (\fa\, \pi_{\cL} (\fb \fc)) + \pi_{\cL} (\fa\, \kappa)\ \ \  (\kappa
\in \cA\cdot \cK)\\
&=& \pi_{\cL} (\fa \, \pi_{\cL} (\fb \fc)) = \rho_{\cL} (\fa) \circ \rho_{\cL}(\fb) (\fc)
\end{eqnarray*}
since $\fa\, \kappa\in \cA\cdot \cK= \ker \pi_{\cL}$ (Proposition \ref{prop:sum}).

We have to prove secondly that $\forall \fa \in \cA$, $\forall \fb, \fc \in
\cB'$,
$ \langle \rho_{\cL} (\fa) \fb, \fc \rangle_{\cL} = { \langle \fb,
\rho_{\cL} ( \fa^{\ast}) \fc \rangle_{\cL}}$:
 \begin{eqnarray*}
\langle \rho_{\cL} (\fa) \fb, \fc \rangle_{\cL} &=& \cL (\fc^{\ast}\pi_{\cL}
(\fa \fb))\\
&=& \cL(\fc^{\ast} \fa \fb) + \cL (\fc^{\ast}\, \kappa) =  \cL(\fc^{\ast} \fa \fb)
\ \ \ \ \ \ \ \ \ \ \ \ \,(\kappa \in K_{\cL} (\cA))\\
&=& \overline{\cL( \fb^{\ast} \fa^{\ast}\fc)} =
\overline{\cL( \fb^{\ast} \pi_{\cL}(\fa^{\ast}\fc))} + 
\overline{\cL( \fb^{\ast} \kappa')} 
\ \ \  (\kappa' \in K_{\cL} (\cA))\\
&=&  \overline{ \cL ( \fb^{*}
\rho_{\cL} ( \fa^{\ast}) \fc )}= \overline{ \langle \rho_{\cL} (
\fa^{\ast}) \fc,  \fb \rangle_{\cL}}= 
{ \langle  \fb,  \rho_{\cL} (\fa^{\ast}) \fc\rangle_{\cL}}.
 \end{eqnarray*} 
This concludes the proof that $\rho_{\cL}$ is a $*$-representation of
$\cA$ on $\cB'$. Since $\cL(\fa)= \langle \rho_{\cL} (\fa) 1,
1\rangle_{\cL} $ for all $\fa\in \cA$, $\rho_\cL$ is the $*$-representation associated with $\cL$.
\end{proof}

\section{Applications}

Let $\pi$ be a $\ast$-representation of $\cA$ on a unitary space $(V,
\langle \cdot, \cdot \rangle)$. For any $v \in V$, the linear functional $L_v$
on $A$ defined by $L_v (\cdot) = \langle \pi (\cdot) v, v \rangle$ is positive and $\cK_v \assign \{a \in \cA : \pi (a) v = 0\}$
is a two-sided $\ast$-ideal of $\cA$. 
Let $\cB$ and $\cC$ be $\ast$-invariant linear subspaces of
$\cA$ such that $1 \in \cB$ and $\cB \subseteq
\cC$. Clearly, $\cK_L ( \cC) = \cK_v \cap \cC$. Then the restriction of $L_v$ to $\cC^2$ is a flat
extension with respect to $\cB$ if and only if $\cC =
\cB + \cK_v \cap C$.

Suppose in addition that $V_v:=\pi(\cA)v$ is finite dimensional. Then $\cA /
\cK_v$ is finite dimensional, so we can choose a finite dimensional
$\ast$-invariant subspace $\cB$ containing $1$ such that $\cA
= \cB + \cK_v$. Then, for any $\ast$-invariant subspace
$\cC$ of $\cA$ which contains $\cB$, the restiction of
$L_v$ to $\cC^2$ is a flat extension with respect to $\cB$.
This provides a large class of examples of flat extensions.

We now develop the three 
applications of Theorem \ref{thm:main} mentioned in the Introduction.

\subsection{Truncated moment problem on cylinder sets}

Let $\cA$ be the polynomial $*$-algebra $\CC[x_1,\dots,x_d,y]$ in
$d+1$ hermitean variables $x_1,\dots,,x_d,y$. 
The algebra $\cA$ is the quotient of the free $*$-algebra
$\cF=\CC\langle x_1,\dots,x_d,y\rangle$ by the commutation relations
$x_{i} x_{j}-x_{j} x_{i}=0$
$x_{i} y-y x_{i}=0$, so that $\delta (\cA)=2$.
For  $k\in\NN$, let $\cA_k$ be the linear span of $x^\alpha \CC[y]$, where $\alpha\in \NN_0^d$,  $|\alpha|\leq k$. 
\begin{proposition}\label{truncatedcylinder} Let $m\in \NN$,  $\cB=\cA_m$ and $\cC=\cA_{m+1}$. Suppose  that $L$ is a positive linear functional on $\cC^2$ which is a flat extension with respect to $\cB$. Then there exist finitely  points $t_1,\dots,t_k\in \RR^d$, where $k\leq {d+1+m \choose m}$, and a positive Borel measure $\mu$ on $\RR^{d+1}$ supported by the set $\cup_{j=1}^k t_j\times \RR$ such that 
$$L(p)=\int_{\RR^{d}\times \RR} p(x,y)\, d\mu(x,y)\quad {\rm for}\quad  p\in \cC^2=\cA_{2m+2}.$$
\end{proposition}
\begin{proof}
Since $\delta (\cA)=2$ and  $\cB^{[1]}=\cA_m^{[1]}=\cA_{m+1}=\cC$, the assumptions of Theorem \ref{thm:main} are fulfilled. Hence $L$ has an extension to a  positive linear functional  $\tilde{L}$ on $\cA$ which is
  a flat extension with respect to $\cB$. Let $L_0$ and $\tilde{L}_0$ denote the restrictions of $L$ and $\tilde{L}$, respectively, to the $*$-subalgebra $\CC[x_1,\ldots,x_d]$ of $\cA$. Then $\tilde{L}_0$ is a flat extension of $L_0$ with respect to $\cB\cap\CC[x_1,\ldots,x_d]$. Therefore, by the theorem of Curto and Fialkow \cite{CF96}, $\tilde{L}_0$ is given by a $k$-atomic measure $\mu_0$ on $\RR^d$, where $k\leq {d+1+m \choose m}$. Let $t_1,\dots,t_k$ denote the atoms of $\mu_0$. For $f\in \cA$, it follows from  the Cauchy-Schwarz inequality that
  \begin{align*}
\big|\tilde{L}\big( \|x-t_1\|^2\cdots \|x-t_k\|^2f\big)\big|^2&\leq \tilde{L}\big(\|x-t_1\|^4\cdots\|x-t_k\|^4\big)\tilde{L}(f^2)\\&= \tilde{L}_0\big( \|x-t_1\|^4\cdots\|x-t_k\|^4\big)\tilde{L}(f^2)= 0.
\end{align*}
Hence $\tilde{L}$ vanishes on the ideal $\cJ$ generated by the polynomial $ \|x-t_1\|^2\cdots\|x-t_k\|^2$.
The zero set of the ideal $\cJ$ is the  set $\{t_1,\dots,t_k\}$. Hence (by \cite[Note 2.1.8]{M}) $\cJ+\sum \cA^2$ is the preorder of the semi-algebraic set $K:=\cup_{j=1}^k t_j\times \RR$ in $\RR^{d+1}$. Since $K$ is a cylinder with finite, hence compact, base set $\cup_{j=1}^k t_j$ in $\RR^d$, it follows from  Corollary 10 in \cite{Sch2003} (see also \cite{KM}) that  $K$ has property (SMP), that is, $\tilde{L}$ is given by a positive Borel measure $\mu$ supported by $K$.
\end{proof}

\subsection{Truncated moment problem for matrices of polynomials}
Let $\cA$ be the unital $*$-algebra of $n\times n$ matrices with
entries from $\CC[x_1,\dots,x_d]$.  Let $e_{ij}$ denote the
corresponding matrix units. We have $e_{ij}^*=e_{ji}$. 
The algebra $\cA$ is the quotient of the free unital $*$-algebra
$\cF=\CC\langle x_{1}, \ldots, x_{d}, e_{11}, \ldots, e_{nn}\rangle$
by the ideal generated by the commutation relations between the variables $x_{i}$, 
the commutation relations $x_{i} e_{jk}-e_{jk} x_{i}=0$ and the 
product relations $e_{ij}e_{jl}=e_{il}$, $e_{ij}e_{kl}=0$ if $k\neq l$.
This implies that $\delta (\cA)=2$.

First we prove a simple well-known lemma.
\begin{lemma}\label{irrerepmatrix}
Let $\rho$ be an irreducible $*$-representation of $\cA$ on a finite dimensional unitary space $V$. Then there exist a unitary operator $U$ of $\CC^d$ onto $V$ and a point $t_0\in \RR^d$ such that $\rho((p_{jk}))=U((p_{jk}(t_0))U^{-1}$ for all matrices $(p_{jk})\in \cA$.
\end{lemma}
\begin{proof}
Since the unit matrix $E$ is the unit element of  $\cA$,   $\rho(E)=I$ by Definition \ref{defrep}. Hence, 
for any $p\in \CC[x_1,\dots,x_d]$ the operator $\rho(p E)$ belongs to
the commutant of $\rho(\cA)$. Therefore, since $\rho$ is irreducible,
there is a complex number $\chi(p)$ such that $\rho(pE)=\chi(p)I$. Clearly,   $p\to \rho(pE)$ is a $*$-homomorphism. Hence the map $p\to \chi(p)$ is a character on the $*$-algebra $\CC[x_1,\dots,x_d]$, so  there is a point $t_0\in \RR^d$ such that $\chi(p)=p(t_0)$ for  $p\in \CC[x_1,\dots,x_d]$.

Let $M_d(\CC)$ denote the $*$-subalgebra of  constant matrices in $\cA.$
For $(p_{jk})\in \cA$ we have $\rho((p_{jk}))=\sum_{j,k} \rho(p_{jk}E)\rho(e_{jk})$. This implies that $\rho(\cA)$ and $\rho(M_d(\CC))$ have the same commutants. Therefore, $\rho_{|M_d(\CC)}$ is an irreducible $*$-representation of the matrix algebra  $M_d(\CC)$ and hence unitarily equivalent to the identity representation. That is, there is a unitary $U$ of  $\CC^d$ onto $V$ such that $\rho(e_{jk})=Ue_{jk}U^{-1}$ for $j,k=1,\dots,d$. Then, for  $(p_{jk})\in \cA$, we derive
\begin{align*}
\rho((p_{jk}))&= \sum_{j,k=1}^d \rho(p_{jk}E)\rho(e_{jk})= \sum_{j,k=1}^d p_{jk}(t_0)Ue_{jk}U^{-1}\\&=U\bigg(\sum_{j,k=1}^d p_{jk}(t_0)e_{jk}\bigg)U^{-1}=U(p_{jk}(t_0))U^{-1}.
\end{align*}
\end{proof}

Fixe $k\in \NN$.  We denote by $\cA_k$ the span of all elements $x^\alpha e_{ij}$, where $|\alpha|\leq k$ and $i,j=1,\dots,n$. Here $x^\alpha e_{ij}$ denotes  the matrix with entry $x^\alpha$ at the $(i,j)$-place and zero otherwise.

\begin{proposition} Suppose that $m\in \NN$,  $\cB=\cA_m$, $\cC=\cA_{m+1}$, and  $L$ is a positive linear functional on $\cC^2$ which is a flat extension with respect to $\cB$. Then there exist   points $t_i\in \RR^d$ and vectors $u_i=(u_{li})\in \CC^d$, $i=1,\dots,r$, $r\in \NN$, such that 
\begin{align}\label{trunvmatrices}
L((p_{jk}))=\sum_{j,k=1}^d \sum_{i=1}^r p_{jk}(t_i)u_{ki}\overline{u_{ji}}\quad {\rm for}\quad  (p_{jk})\in \cC^2=\cA_{2m+2}.
\end{align}
\end{proposition}
 \begin{proof}
As $\delta (\cA)=2$ and $\cB^{[1]}=\cA_m^{[1]}=\cA_{m+1}=\cC$, by Theorem \ref{thm:main}, $L$ has a flat extension  to a positive linear functional $\tilde{L}$ on the whole algebra $\cA$.
Let $\rho_{\tilde{L}}$ denote the GNS representation of  $\tilde{L}$.  Since $\rho_{\tilde{L}}$ acts on a subspace of the finite dimensional space $\cB$,  it is an orthogonal direct sum of finite dimensional irreducible $*$-representations $\rho_i$ acting on unitary spaces $V_i$, $i=1,\dots,r$. Each representation $\rho_i$ is of the form described in Lemma \ref{irrerepmatrix}. For $(p_{jk})\in \cA$, we then obtain 
$$
\tilde{L}((p_{jk}))=\langle \rho_{\tilde{L}}((p_{jk}))v,v\rangle_V=\sum_{i=1}^r \langle \rho_i((p_{jk}))v_i,v_i\rangle_{V_i}=
\sum_{i=1}^r \sum_{j,k=1}^d p_{jk}(t_i)u_{ki}\overline{u_{ji}}
$$
which implies (\ref{trunvmatrices}), since $\tilde{L}$ is an extension of $L$.
Here $v$ is the direct sum of vectors $v_i$. Further, $u_i=U_i^{-1}v_i$ and  $t_i$ denotes the point and $U_i$ is the unitary from   Lemma \ref{irrerepmatrix} applied to the irreducible representation $\rho_i$.
\end{proof}
\subsection{Truncated moment problem for enveloping algebras}

Let $\fg$ be real finite dimensional Lie algebra. Recall that the
universal enveloping algebra $\cA \assign \cE ( \fg)$
of $\fg$ is a complex unital $\ast$-algebra with involution
determined by the requirement $(iy)^{\ast} = iy$ for $y \in \fg$. Fix
a basis $\{y_1, \ldots, y_d \}$ of the real vector space $\fg$. Then
there are real numbers $c_{jkl}$, $j, k, l = 1, \ldots, d$, the structure
constants of the Lie algebra $\fg$, such that
\[ [y_j, y_k] = \sum_l c_{jkl} y_l . \]
The $\ast$-algebra $\cE ( \fg)$ is the quotient algebra of
the free $\ast$-algebra $\cF (x_1, \ldots,, x_d)$ be the two-sided
$\ast$-ideal $\cI$ generated by the set $\cI_G$ of elements
\begin{equation}
  x_j x_k - x_k x_j - \sum_l c_{jkl} i x_l, ~ ~ l = 1, \ldots, d,
  \label{idealgen}
\end{equation}
where the quotient map $\sigma$ is given by $\sigma (x_j) = iy_j$, $j = 1,
\ldots, d$. We deduce that $\delta (\cA)\leq 2$.
Further, by the Poincare-Birkhoff-Witt theorem, the set
\[ \{y_1^{n_1} \ldots y_d^{n_d} ; (n_1, \ldots, n_d) \in N_0^d \} \]
is a vector space basis of $\cE ( \fg)$. For $m \in
\NN$, let $\cA_m$ denote the linear span of elements $y_1^{n_1}
\ldots y_d^{n_d}$, where $n_1 + \ldots + n_d \leq m$.

\begin{proposition}\label{tmpenveloping}
  Let $\cB = \cA_m$ and $\cC = \cA_{m + 1}$.
  Suppose that $L$ is a positive linear functional on $\cC^2$ which is a flat
  extension with respect to $\cB$. Then $L$ has an extension to
  a positive linear functional $\tilde{L}$ on $\cA$ and $\tilde{L}$
  is a flat extension with respect to $\cB$.
\end{proposition}

\begin{proof}
Since the ideal
  $\cI$ is generated the quadratic elements in (\ref{idealgen}),
  $\delta (\cA)\leq 2$ and $\cB^{[1]}=\cA_m^{[1]}=\cA_{m+1}=\cC$, the
  assumptions of Theorem  \ref{thm:main} are fulfilled which gives the assertion.
\end{proof}

Let $\cR_{\rm fin} ( \fg)$ denote the family of all
$\ast$-representations of the $\ast$-algebra $\cE ( \fg)$
acting on \tmtextit{finite dimensional} Hilbert spaces. By Theorem  \ref{thm:main}, the GNS representation $\rho_{\tilde{L}}$ associated with the positive functional $\tilde{L}$ on $\cE ( \fg)$ acts on a subspace of the finite dimensional space $\cB$. Hence $\rho_{\tilde{L}}$ belongs to $\cR_{\rm fin} ( \fg)$ and there is a vector $v$ in the representation space of $\rho_{\tilde{L}}$ such that 
\begin{align}\label{envetmp}
L (a) = \langle \rho_{\tilde{L}} (a) v, v \rangle\quad {\rm for}\quad 
a\in \cC^2=\cA_{2m+2}.
\end{align}
%This equation can be considered as a solution of the truncated non-commutative moment problem for the enveloping algebra $\cE ( \fg)$. 
Let $G$ be the simply connected Lie group which has the Lie algebra $\fg$. Then the $*$-representation $\rho_{\tilde{L}}$ of $\cE ( \fg)$ exponentiates to a unitary representation $U$ of $G$. That is, we have $\rho_{\tilde{L}}=dU$ and
equation (\ref{envetmp}) can be considered as a solution of a truncated non-commutative moment problem for the enveloping algebra $\cE ( \fg)$, see e.g. \cite{Sch1991}. Note that it may happen that there is no non-trivial finite dimensional $*$-representation of $\cE ( \fg)$. However, if the Lie group $G$ is {\it compact}, all irreducible unitary representations of $G$ are finite dimensional, so there is a rich theory of truncated non-commutative moment problems for  $\cE ( \fg)$.

\medskip
\noindent{}Bernard Mourrain\\
Inria  Sophia Antipolis M\'editerran\'ee, BP 93, 06902 Sophia Antipolis, France.\\
E-mail address: \texttt{Bernard.Mourrain@inria.fr}

\medskip
\noindent{}Konrad Schm\"udgen\\
Universit\"at Leipzig, Mathematisches Institut, Augustusplatz 10/11, D-04109 Leipzig, Germany\\
E-mail address: \texttt{schmuedgen@math.uni-leipzig.de}

\end{document}